\documentclass[twoside,leqno,11pt]{article}
\usepackage{revstat}
\usepackage{epstopdf} 
\usepackage{amsmath}
\usepackage{graphicx} 
\usepackage{subcaption}
\begin{document}
\title{On The Robustness of Epsilon Skew Extension for Burr III Distribution on Real Line
\footnote{The opinions expressed in this text are those of the authors and do not necessarily reflect the views of any organization.}}
\renewcommand{\titleheading}
             {Using the ``revstat.sty" Package}  
\author{\authoraddress{Mehmet Niyazi \c{C}ankaya}
                      {Department of Statistics, Faculty of Arts and Science
                       U\c{s}ak University,\\
                       U\c{s}ak
                       \ (mehmet.cankaya@usak.edu.tr)}
\\
        \authoraddress{Abdullah Yal\c{c}{\i}nkaya}
                      {Department of Statistics, Faculty of  Science
                       Ankara University,\\
                        \ (ayalcinkaya@ankara.edu.tr)}
                       \\
        \authoraddress{\"{O}mer Alt{\i}nda\v{g}}
                      {Department of Statistics, Faculty of Arts and Science
                       B\.{i}lec\.{i}k \c{S}eyh Edebal\.{i} University,\\
                         \ (ayalcinkaya@ankara.edu.tr)}
                       \\
        \authoraddress{Olcay Arslan}
                      {Department of Statistics, Faculty of  Science
                      Ankara University ,\\                     
                       \ (oarslan@ankara.edu.tr)}
}
\renewcommand{\authorheading}
             {Mehmet Niyazi \c{C}ankaya \ and \ Abdullah Yal\c{c}{\i}nkaya \ and \ \"{O}mer Alt{\i}nda\v{g} \ and \ Olcay Arslan }  

\maketitle

\begin{abstract}
The Burr III distribution is used in a wide variety of fields of lifetime data analysis, reliability theory, and financial literature, etc.
It is defined on the positive axis and has two shape parameters, say $c$ and $k$. These shape parameters make the distribution quite flexible.
They also control the tail behavior of the distribution. In this study, we extent the Burr III distribution to the real axis and also add
a skewness parameter, say $\varepsilon$, with epsilon-skew extension approach.  When the parameters $c$ and $k$ have a relation such that
$ck \approx 1 $ or $ck < 1 $, it is skewed unimodal. Otherwise, it is skewed bimodal with the same level of peaks on the negative and positive sides of real line.
Thus, ESBIII distribution can capture fitting the various data sets even when the number of parameters are three. Location and scale form of this distribution are also given.
Some distributional properties of the new distribution are investigated. The maximum likelihood (ML) estimation method for the parameters
of ESBIII is considered. The robustness properties of ML estimators are studied and also tail behaviour of ESBIII distribution is examined.
The applications on real data are considered to illustrate the modeling capacity of this distribution
in the class of bimodal distributions.
\end{abstract}

\begin{keywords}
Asymmetry; Burr III distribution; bimodality; epsilon skew; robustness; unimodality.
\end{keywords}

\begin{ams}
49A05, 78B26.
\end{ams}

\mainpaper  

\section{Introduction}

It is widely accepted that using flexible and heavy-tailed distributions is important to fit the data sets for practitioners in
application science. \cite{Burr42} proposed twelve different distributions which are quite flexible and useful to model the data sets.
Especially, Burr XII distribution is generally used as a flexible distribution \cite{Burr42,Burr68,Burr73,Rodriquez77,ali2015pak,Wingo83}.
However, Burr III distribution has a wider skewness and kurtosis region than that of Burr XII distribution \cite{Tadikamalla80}.
Burr III distribution has two shape parameters $c$ and $k$ which give the flexibility to the distribution.
Additionally, it has an ability that produces L-shaped or unimodal distribution when $ck \leq 1$ or $ck>1$
are satisfied, respectively \cite{Lindsay96}. Burr III distribution is defined on the positive axis,
so it is widely used in reliability studies \cite{wingo83,wang96,zimmer98,Zoraghietal12}, in the forest study
\cite{Goveetal08,Lindsay96} and in the hydrology study  \cite{hao2009entropy}.
It is also used in the quality and control engineering, reliability analysis  \cite{Shaoetal08,abd2012bayesian,azimi2013bayesian}.


In real life, most of the data sets are generally measured on the real axis and they may also have asymmetry as well.
There are well known distributions in the literature such as normal, student $t$, Laplace.
These are all symmetric distributions. Recently, many researchers have dealt with generalizing
the existing distributions to capture the data sets having asymmetry. There are various generalization
techniques which extent a symmetric distribution to asymmetric form. \cite{Azzalini85} proposed an
approach to construct skewed distribution and introduced firstly skew normal (SN) distribution.
Also \cite{AzzaliniCapi03} presented skew $t$ distribution. One can refer to \cite{Gentonbookskew04}
for Azzalini type skew distributions. \cite{Mud00} introduced so-called epsilon-skew method to
generalize the symmetric distributions to asymmetric form with a skewness parameter, say $\varepsilon$
such that $|\varepsilon|<1$. In that study, they presented the epsilon-skew normal (ESN) distribution.
If the skewness parameter $\varepsilon$ is chosen as zero, it reduces to the normal distribution.
From now on, this method has been applied to several well known distributions. \cite{Elsalloukhetal05}
considered the epsilon-skew extension of the exponential power (ESEP) distribution.
Special cases of the ESEP distribution reduce the ESN distribution \cite{Mud00} and epsilon-skew Laplace
(ESL) distribution \cite{Elsalloukh08}. The scale variant of ESEP distribution was considered
by \cite{Gomez07,ArslanGenc09}. \cite{Hassan13} took into account epsilon-skew extension of the reflected gamma distribution.
A comprehensive content about asymmetric distributions which include wide family of symmetric distributions as a special case can be found in \cite{ArelGom05}.

It should be noted that all the above distributions are unimodal skewed. However, one can encounter bimodal data sets in real life problems.
So, modeling the bimodality is another challenging problem. For this purpose, one can use mixed distributions.
However, there have been proposed bimodal distributions to model the bimodality directly.  \cite{Hassan14,Andrade16,AndradeRat16,Cankayaetal15,Dexter15,Elal09,Donatella10,Genc13,Gui14,Has10,Hasel16,Rocha13,Cooray13,Jama11,Sham13,Rego12,Elal10,Arellanoetal10,Gomez11,Ahmed08,San08,Fam04,Eu02} proposed the distributions in the class of bimodality.

In this paper, we first reflect the Burr III distribution to real axis and consider its epsilon-skew extension. Anymore, we call it as the epsilon-skew Burr III
(ESBIII) distribution. If the skewness parameter $\varepsilon \neq 0$, it can be unimodal skewed or bimodal skewed with same level of peaks, depending on the shape parameters.
The new distribution is not only unimodal, bimodal and skew but also candidate for efficient modelling when compared with ESEP distribution, because of the two
shape parameters in Burr III distribution. The scale version of ESEP is generalized $t$ distribution proposed by \cite{ArslanGenc09} and it has three parameters,
shape peakedness, tail and skewness parameters. In $ESBIII$, the number of parameters are three, however, it has the shape parameters $c$ and $k$
which help not only shape peakedness and tail behaviour but also the behaviour between peakedness and tail behaviour all together.
The order statistics distributions from ESBIII can be explicitly defined as an another advantage due to having an exact form of cumulative distribution function.
Since the skewed form of these two distributions are in same type, it is said that the the responsibility of parameter $\varepsilon$ are same for these two distributions.
 The robustness property of this distribution is wide when compared with well known Student $t$ and its generalized skew version considered by \cite{ArslanGenc09}.
 We also propose the location-scale transformation of ESBIII distribution.

Remainder of the paper is set up as follows. In Section 2, we present epsilon-skew extension of Burr III distribution. In Section 3,
we study some basic properties of the ESBIII distribution.  In Section 4,
we obtain the maximum likelihood estimators (MLE) for the parameters of the ESBIII distribution.
Section $5$ investigates of the robustness property of MLE. In Section 6, some examples are given to show the performance of the ESBIII. The last section considers the conclusions.

\section{The Epsilon Skew Burr III Distribution}\label{B3definiton}

Consider the random variable $Z$ has a Burr III distribution with the density function
\begin{equation}
g(z)=ck z^{-(c+1)} (1+z^{-c})^{-(k+1)}.
\end{equation}
and the distribution function
\begin{equation}
G(z)=(1+z^{-c})^{-k}.
\end{equation}
where $z>0$ and the shape parameters $c>0$ and $k>0$ are real-valued parameters that identify the characteristic of the distribution such as dispersion,
peakedness, tail thickness. As noted earlier, when $ck \leq 1$ the Burr III distribution is L-shaped, when $ck>1$ the Burr III distribution is unimodal.\\

The following theorem states the epsilon skew extension of the Burr III distribution.
\begin{theorem}\label{B3theorem}
Let $Z$ be a continuous random variable distributed as a Burr III with the parameters $c$ and $k$. Let $U$ be
a discrete random variable with the following values with their probabilities,
\begin{equation}\label{Trv}
U=\left\{
  \begin{array}{ll}
    1+\varepsilon, & \frac{1+\varepsilon}{2} \\
    -(1-\varepsilon), & \frac{1-\varepsilon}{2}
  \end{array}
\right.
\end{equation}
\noindent where $\varepsilon \in (-1,1)$. Assume that $Z$ and $U$ are independent. Then, the distribution of the random variable
\begin{equation}\label{xvariable}
X=Z U
\end{equation}
\noindent  will have the following distribution function
\begin{equation}\label{dfESB3}
    F(x)=\left\{
  \begin{array}{ll}
   \frac{1-\varepsilon}{2}+(\frac{1+\varepsilon}{2})(1+(\frac{x}{1+\varepsilon})^{-c})^{-k} , & x \geq 0 \\
   (\frac{1-\varepsilon}{2})(1-(1+(\frac{-x}{1-\varepsilon})^{-c})^{-k}) , & x<0
  \end{array}
\right.
\end{equation}
and will have the following density function
\begin{equation}\label{pdfESB3}
f(x)=\frac{ck}{2} \bigg(\frac{sign(x)x}{1+sign(x)\varepsilon}\bigg)^{-(c+1)}\bigg(1+\bigg(\frac{sign(x)x}{1+sign(x)\varepsilon}\bigg)^{-c}\bigg)^{-(k+1)} , x \in \mathbf{R\backslash{0}}
\end{equation}
\noindent with the shape parameters $c>0$, $k>0$ and the skewness parameter $\varepsilon \in (-1,1)$.
\end{theorem}

\begin{proof}
Using the distribution function technique, the CDF of $X$ for $x<0$ is
\begin{equation}\label{p1esb3}
    F_{X}(x)=P(ZU\leq x)=\frac{1-\varepsilon}{2}\bigg\{1-F_{Z}\bigg(\frac{-x}{1-\varepsilon}\bigg)\bigg\}
\end{equation}
and for $x \geq 0$ is
\begin{equation}\label{p1esb3}
    F_{X}(x)=P(ZU\leq x)=\frac{1-\varepsilon}{2}+\frac{1+\varepsilon}{2}\bigg\{F_{Z}\bigg(\frac{x}{1+\varepsilon}\bigg)\bigg\}.
\end{equation}
Taking the derivative of equation~\eqref{dfESB3} with respect to $x$ yields the density function of $X$ given in equation~\ref{pdfESB3}.
\qed
\end{proof}

\begin{definition}
The distribution of the random variable $X$ with the density function in equation \ref{pdfESB3} is called a Epsilon Skew Burr III  ($ESBIII$) distribution.
\end{definition}

\begin{definition}
Suppose that $X \sim ESBIII(c,k,\varepsilon)$. Then, the random variable $Y=X\sigma + \mu$, $\sigma>0$, $\mu \in \mathbf{R}$ will have $ESBIII$ distribution
with the following density function ($Y \sim ESBIII(\mu,\sigma,c,k,\varepsilon)$)

\begin{equation}\label{locationscaleESBIII}
f(y)=\frac{ck}{2\sigma} \bigg(\frac{sign(\frac{y-\mu}{\sigma})\frac{y-\mu}{\sigma}}{1+sign(\frac{y-\mu}{\sigma})\varepsilon}\bigg)^{-(c+1)}\bigg(1+\bigg(\frac{sign(\frac{y-\mu}{\sigma})\frac{y-\mu}{\sigma}}{1+sign(\frac{y-\mu}{\sigma})\varepsilon}\bigg)^{-c} \bigg)^{-(k+1)} , y \in \mathbf{R\backslash{0}}.
\end{equation}
\end{definition}

The $ESBIII$ distribution has five parameters; $\mu$, $\sigma$, $\varepsilon$, $c$ and $k$. The parameters $\mu$ and $\sigma$
determine the location and the scale of the distribution, respectively. The parameter $\varepsilon$ controls the skewness
proportion on the left and right hand side of the location. The parameters $c$ and $k$ control the shape of the density together.
It is noted that when $\varepsilon=0$, it is symmetric unimodal or symmetric bimodal.\\


\subsection{On the flexibility of ESBIII distribution}

Choosing special cases of the values of shape parameters $c$ and $k$ gives the flexible probability density functions (PDF).
For example, the shape of ESBIII distribution for some shape values can be similar to not only the epsilon skew Laplace distribution
but also epsilon skew t distribution. In this context, it is considered that the new distribution can model data sets which cannot
be precisely modelled by epsilon skew Laplace and epsilon skew t distributions, as exemplified by section \ref{applications}. The ESBIII has the properties of both leptokurtic
and platikurtic distributions. When $ck \approx 1$ or $ck < 1$, the PDF is unimodal. For $ck \approx 1$ or $ck < 1$ case,
when the difference between $c$  and $k$ is extremely large, the shape of distribution is platikurtic,
when the difference between $c$  and $k$ is extremely small, the shape of distribution is leptokurtic. The distributions
proposed by \cite{mineo2005software} and the references in \cite{Cankayaetal15} are in platikurtic-leptokurtic family as well.
When the values of $ck$ are close to $1$ from left hand side of $1$, the distribution has a peak being a needle. When $ck>1$,
the PDF is bimodal. The new distribution has also a skewness parameter. Thus, the base length of each group in the negative
and positive side of function becomes different. In Figure 1, some graphs are given to understand the behavior of the newly defined function for the different values of parameters.
The first and third columns in Figure $1$ represent the role of skewness parameter $\varepsilon$.
\begin{figure}[!htb]
\advance\leftskip-1.8cm
\includegraphics[width=150mm, scale=.45]{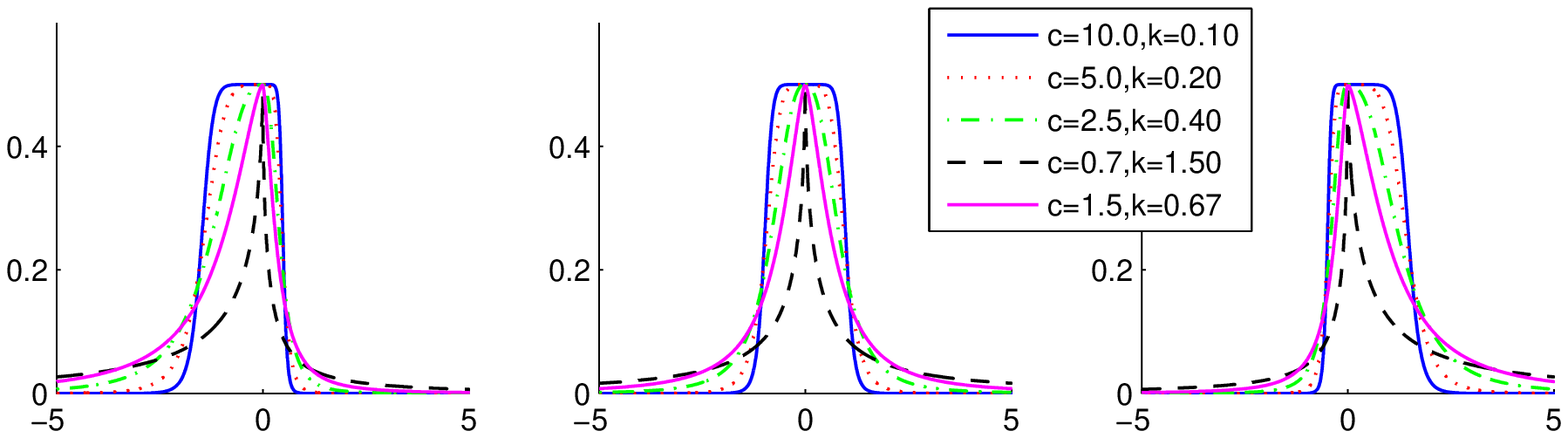}
\includegraphics[width=150mm, scale=.45]{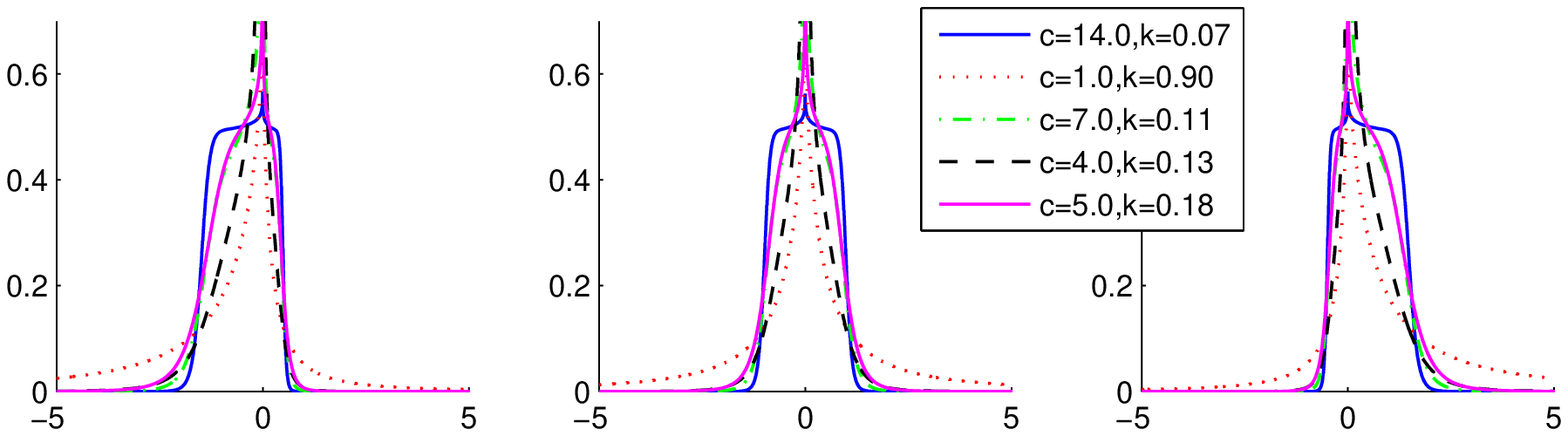}
\includegraphics[width=150mm, scale=.45]{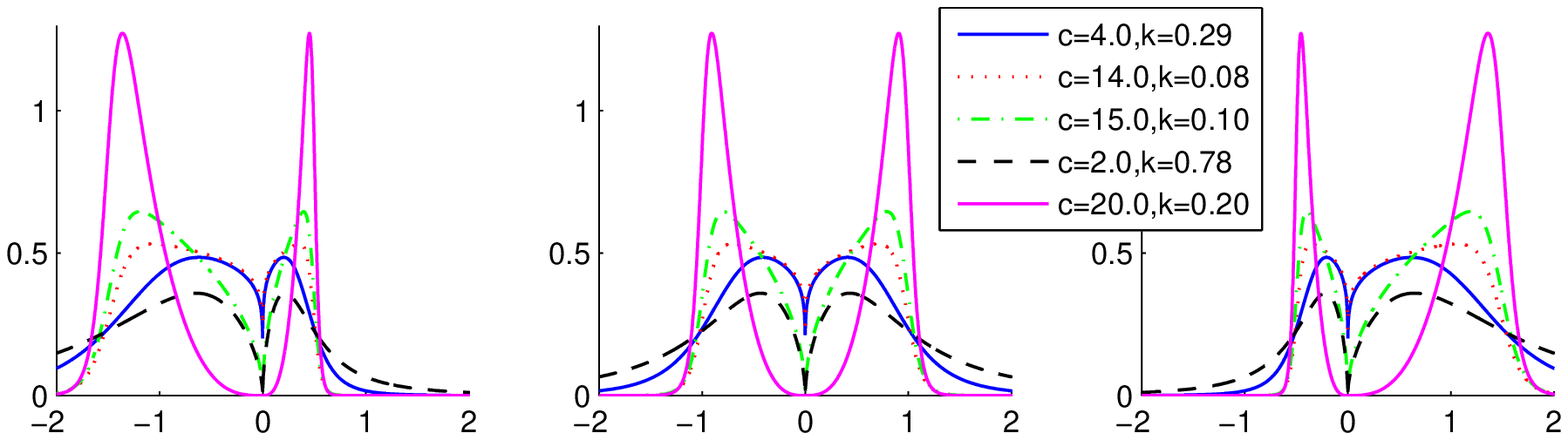}
 \caption{Examples of the PDF of the $ESBIII$ distribution for the different values of parameters, $\varepsilon=-0.5$ (left side), $\varepsilon=0$ (center) and $\varepsilon=0.5$ (right side).}
  \label{fig:pdfexample}
\end{figure}
\begin{table}
\begin{center}
\caption{The skewness and the kurtosis values of $ESBIII$ distribution for the different values of parameters}
\label{tab:1}
\begin{tabular}{p{2.5cm} p{1.3cm} p{1.3cm} p{1.3cm} p{1.3cm} p{1.3cm}}
  \hline
  $c=20, k=0.20$ & $\varepsilon=0$ & $\varepsilon=0.2$ & $\varepsilon=0.4$ & $\varepsilon=0.6$ & $\varepsilon=0.8$  \\
  \hline
  Skewness & 0.0000 & 0.7453 & 1.0945 & 1.1553 & 1.1167 \\
  Kurtosis & 1.1600 & 1.3021 & 1.4448 & 1.4071 & 1.2854 \\
  \hline
  \hline
  $c=14, k=0.07$ & $\varepsilon=0$ & $\varepsilon=0.2$ & $\varepsilon=0.4$ & $\varepsilon=0.6$ & $\varepsilon=0.8$  \\
  \hline
  Skewness & 0.0000 & 0.9398 & 1.3801 & 1.4568 & 1.4082 \\
  Kurtosis & 1.9481 & 2.1866 & 2.4262 & 2.3631 & 2.1587 \\
  \hline
  \hline
  $c=10, k=0.10$ & $\varepsilon=0$ & $\varepsilon=0.2$ & $\varepsilon=0.4$ & $\varepsilon=0.6$ & $\varepsilon=0.8$  \\
  \hline
  Skewness & 0.0000 & 0.9591 & 1.4084 & 1.4866 & 1.4370 \\
  Kurtosis & 2.0701 & 2.3236 & 2.5782 & 2.5111 & 2.2939 \\
  \hline
  \hline
  $c=7, k=0.111$ & $\varepsilon=0$ & $\varepsilon=0.2$ & $\varepsilon=0.4$ & $\varepsilon=0.6$ & $\varepsilon=0.8$  \\
  \hline
  Skewness & 0.0000 & 1.0904 & 1.6013 & 1.6902 & 1.6338 \\
  Kurtosis & 2.8707 & 3.2222 & 3.5752 & 3.4822 & 3.1809 \\
  \hline
  \hline
  $c=5, k=0.20$ & $\varepsilon=0$ & $\varepsilon=0.2$ & $\varepsilon=0.4$ & $\varepsilon=0.6$ & $\varepsilon=0.8$  \\
  \hline
  Skewness & 0.0000 & 1.1760 & 1.7271 & 1.8230 & 1.7622 \\
  Kurtosis & 4.2853 & 4.8100 & 5.3371 & 5.1982 & 4.7485 \\
  \hline
\end{tabular}
\end{center}
\end{table}

In Table $1$, some values of the skewness and kurtosis are given for the different values of parameters. When the distribution
is the positive side skew, the sign of skewness parameter is positive. For negative side skewness, the skewness parameter is negative.
For both case, the kurtosis of distribution is positive. Since the skewness and kurtosis values are greater, it is obvious that the distribution is flexible.

\section{Some Properties of ESBIII Distribution}
The moments of $X \sim ESBIII(c,k,\varepsilon)$ are given by the following theorem.
\begin{theorem}\label{rthmoment}	
Let $X \sim ESBIII(c,k,\varepsilon)$. The $rth$, $r \in \mathbf{R}$, noncentral moments are given by
\begin{equation}
\mathbf{E}(X^r)=\frac{k}{2} \mathbf{B}(1-\frac{r}{c},\frac{r}{c}+k) \{(1+\varepsilon)^{r+1}+(-1)^{-r}(1-\varepsilon)^{r+1}\}, c>r.
\end{equation}
\end{theorem}

\begin{proof}
The $rth$ moment of $X$ can be easily obtained making the transformation $x=(1+\varepsilon)u^{-1/c}$ for negative part of integral
and the transformation $x=-(1-\varepsilon)u^{-1/c}$ for positive part of integral and using the following beta function
\begin{equation}\label{betafunc}
    \int_{0}^{\infty}u^{-r/c}(1+u)^{-k-1}du=\mathbf{B}(1-\frac{r}{c},\frac{r}{c}+k),
\end{equation}
where $c>r$.
\qed
\end{proof}

\begin{corollary}
Let $X\sim ESBIII(c,k,\varepsilon)$. The expected value of $X$ is
\begin{equation}
   \mathbf{E}(X)=\frac{k}{2} \mathbf{B}(1-\frac{1}{c},\frac{1}{c}+k) 4\varepsilon,~~ c>1.
\end{equation}
\end{corollary}

\begin{corollary}
Let $X\sim ESBIII(c,k,\varepsilon)$. The variance value of $X$ is
\begin{equation}
   Var(X)=\frac{k}{2} \mathbf{B}(1-\frac{2}{c},\frac{2}{c}+k)(2+6\varepsilon^2)-k^2 \mathbf{B}^2(1-\frac{1}{c},\frac{1}{c}+k)4\varepsilon^2 ,~~ c>2.
\end{equation}
\end{corollary}

\begin{theorem}\label{CharaFunwithoutmusg}	
Let $X \sim ESBIII(c,k,\varepsilon)$. The characteristic function is given by
\begin{equation}
\mathbf{E}(exp(itX))= \sum_{r=0}^{\infty}  \frac{k (it)^r}{2r!} \mathbf{B}(1-\frac{r}{c},\frac{r}{c}+k) \{(1+\varepsilon)^{r+1}+(-1)^{-r}(1-\varepsilon)^{r+1}\}, c>r.
\end{equation}
\end{theorem}

\begin{proof}
The Taylor expansion of the function
$exp(itx) = \sum_{r=0}^{\infty} \frac{(itx)^r}{r!}$
can be used to calculate the integral $E(exp(itX))$.
\end{proof}

\begin{corollary}\label{CharaFunmusg}	
Let $Y \sim ESBIII(\mu,\sigma,c,k,\varepsilon)$. The characteristic function (CF) for $\mu$ and $\sigma$ is given by
\begin{equation}
\mathbf{E}(exp(itY))= exp(it \mu) \sum_{r=0}^{\infty}  \frac{k (it \sigma)^r}{2r!} \mathbf{B}(1-\frac{r}{c},\frac{r}{c}+k) \{(1+\varepsilon)^{r+1}+(-1)^{-r}(1-\varepsilon)^{r+1}\}.
\end{equation}
\end{corollary}

It is well known that $r!=\Gamma(r+1)$. Thus, the fractional form of CF will be proposed. It will be beneficial for the non integer order values for quantitative $r$, $r>0$.

Since the CDF is tractable, it is easy to propose the quantile function $x_q$ of $ESBIII$.

\begin{equation}\label{dfESB3}
F^{-1}(p)=x_q=\left\{
  \begin{array}{ll}
   (1+\varepsilon)\{ [\frac{2F(x)}{1+\varepsilon}-\frac{1-\varepsilon}{1+\varepsilon} ]^{-1/k} -1 \}^{-1/c} , & x \geq 0 \\
   -(1-\varepsilon)[(1-\frac{2F(x)}{1-\varepsilon})^{-1/k}-1]^{-1/c} , & x<0.
  \end{array}
\right.
\end{equation}

It is commonly used the R\'{e}nyi entropy for a random variable $X$ with PDF $f$. It is defined to be
\begin{equation}\label{Renyientropy}
 \frac{1}{1-\alpha}\log\int_{-\infty}^{\infty}f^{\alpha}(x)dx.
\end{equation}
The following theorem is the R\'{e}nyi entropy for the ESBIII distribution.
\begin{theorem}
 Let the random variable $X$ have an $ESBIII(0,1,c,k,\varepsilon)$ distribution. The R\'{e}nyi entropy of $X$ is given by
\begin{equation}\label{RenyientropyESBIII}
 \frac{1}{1-\alpha}\log(I_1+I_2),
\end{equation}
where $\alpha>0$ and $\alpha \neq 1$,
\begin{equation}
I_1=(1-\varepsilon)(\frac{ck}{2})^{\alpha}\Gamma(\alpha(k+1)-\alpha(1+1/c)+1/c) \frac{\Gamma(\alpha(1+1/c)-1/c)}{c\Gamma(\alpha(k+1))}
\end{equation}
 and
 \begin{equation}
 I_2=(1+\varepsilon)(\frac{ck}{2})^{\alpha}\Gamma(\alpha(k+1)-\alpha(1+1/c)+1/c) \frac{\Gamma(\alpha(1+1/c)-1/c)}{-c\Gamma(\alpha(k+1))}.
 \end{equation}
 \end{theorem}
\begin{proof}
The positive side of PDF $f$ with R\'{e}nyi entropy can be written as
$f_{+}^{\alpha}(x)=(\frac{ck}{2})^{\alpha}(\frac{x}{1+\varepsilon})^{-\alpha(c+1)}\bigg(1+(\frac{x}{1+\varepsilon})^{-c}\bigg)^{-\alpha(k+1)}$.
To get the result of the integral, $\frac{x}{1+\varepsilon}$ is taken to be $u$. After some straightforward calculation, the integral can be rewritten as following form,
\begin{eqnarray*}
I_2&=&  (1+\varepsilon)(\frac{ck}{2})^{\alpha}\int_{0}^{\infty} u^{-\alpha(c+1)}(1+u^{-c})^{-\alpha(k+1)}du \\ \nonumber
  &=& (1+\varepsilon)(\frac{ck}{2})^{\alpha} \Gamma(\alpha(k+1)-\alpha(1+1/c)+1/c) \frac{\Gamma(\alpha(1+1/c)-1/c)}{-c\Gamma(\alpha(k+1))}.
\end{eqnarray*}
The same procedure can be done for the negative side of PDF $f$.
\qed
\end{proof}

\section{ML Estimators of Parameters in ESBIII and an Iterative Reweighting Approach for Computation}\label{MLELSESBIII}
Let $x_1,x_2,...,x_{n}$ be a random sample of size $n$ from a $ESBIII$ distributed population.
We would like to estimate the unknown parameters $\mu,\sigma,c,k,\varepsilon$. In this sense, we can prefer to use MLE and then the log-likelihood function is
\begin{eqnarray}
 l &=& n \log(\frac{ck}{2\sigma})-(c+1)\sum_{i=1}^{n} \log(\frac{ sign(\frac{x_{i}-\mu}{\sigma}) (x_{i}-\mu)}{\sigma [1+sign(\frac{x_{i}-\mu}{\sigma}) \varepsilon]}) \\ \nonumber
  && -(k+1)\sum_{i=1}^{n} \log(1+(\frac{ sign(\frac{x_{i}-\mu}{\sigma}) (x_{i}-\mu)}{\sigma [1+sign(\frac{x_{i}-\mu}{\sigma}) \varepsilon]})^{-c}),\\ \nonumber
\end{eqnarray}
\noindent where $sign$ is the signum function and $sign(0)=1$.

 The maximum likelihood estimators of the parameters $\mu,\sigma,c,k$ and $\varepsilon$ will be solution of the following  equations

\begin{eqnarray}\label{nonlineareqmuu}
  \frac{\partial l}{\partial \mu}  &=& (c+1)\sum_{i=1}^{n} (x_{i}-\mu)^{-1}-c(k+1)\sum_{i=1}^{n} \frac{(x_{i}-\mu)^{-1}}{1+z_{i}^{c}}=0,
\end{eqnarray}

\begin{eqnarray}\label{nonlineareqsigma}
  \frac{\partial l}{\partial \sigma}  &=& \frac{nc}{\sigma}-c(k+1)\sum_{i=1}^{n} \frac{\sigma^{-1}}{1+z_{i}^{c}}=0,
\end{eqnarray}

\begin{eqnarray}\label{nonlineareqc}
   \frac{\partial l}{\partial c} &=& \frac{n}{c}-\sum_{i=1}^{n} \log(z_{i}) +(k+1)\sum_{i=1}^{n} \frac{\log(z_{i})}{1+z_{i}^{c}}=0,
\end{eqnarray}

\begin{eqnarray}\label{nonlineareqk}
    \frac{\partial l}{\partial k}  &=&  \frac{n}{k}-\sum_{i=1}^{n} \log\big(1+z_{i}^{-c}\big)=0,
\end{eqnarray}

\begin{eqnarray}\label{nonlineareqeps}
  \frac{\partial l}{\partial \varepsilon}  &=& (c+1)\sum_{i=1}^{n} (s_{i}+\varepsilon)^{-1}-c(k+1)\sum_{i=1}^{n} \frac{(s_{i}+\varepsilon)^{-1}}{1+z_{i}^{c}}=0,
\end{eqnarray}

\noindent where $s_{i}=sign(\frac{x_{i}-\mu}{\sigma}) $ and $z_{i}=\frac{s_{i}(x_{i}-\mu)}{\sigma [1+s_{i} \varepsilon]}$.\\

Since it is not possible to solve these equations analytically, the numerical methods should be used to obtain the $ML$ estimators.
Eqs. (\ref{nonlineareqmuu}) - (\ref{nonlineareqeps}) are solved in simultaneous time as an iteratively way, using the nonlinear solving
techniques in MATLAB 2013a. An example for solving an equation in Eqs. (\ref{nonlineareqmuu}) - (\ref{nonlineareqeps}) is as follow,

\begin{equation}\label{compustep}
 \hat{\mu}^{(i+1)} := h(\mu,\hat{\sigma}^{(i)},\hat{c}^{(i)},\hat{k}^{(i)},\hat{\varepsilon}^{(i)}).
\end{equation}

As it is seen from equation \eqref{compustep}, the initial values of other parameters are provided.
The function $h$ representing the derivative of log-likelihood function with respect to parameter $\mu$ is solved
according to parameter $\mu$ via the nonlinear solving techniques in MATLAB 2013a. The same procedure was conducted for the other parameters of ESBIII.
While conducting procedure for solving, each step $i$ in equation \eqref{compustep} is computed until the prescribed value provided to finalize the
iterative algorithm is reached. The same procedure was conducted for the estimations of parameters of other distributions used to show
the fitting competence of ESBIII distribution. Kolmogorov-Smirnov statistic (KS) is a good tool to test whether or not the initial
values of parameters can be selected appropriately. The different initial values were tried until high value of KS is reached.
The approach can be more preferable when the computational issue for solving the nonlinear equations simultaneously has been arose.

\section{On The Robustness Properties of ESBIII Distribution}

\subsection{Robustness of ML estimators of parameters in ESBIII distribution}

ML estimators will be robust if the norm of influence function (IF) of ML estimators is bounded. IF is defined to be a linear transformation of the score vector.
As given in \cite{Hampel86}, IF for $\hat{\theta}$ is

\begin{equation}\label{IFdefinition}
IF(x;\hat{\theta})=-(E[\partial \psi(x;\boldsymbol \theta)/\partial \theta])^{-1}\psi(x;\boldsymbol \theta).
\end{equation}

\noindent It means that IF is proportional to the score functions, where $\theta = (\mu,\sigma,c,k,\varepsilon)$ is a vector of parameters.
$\psi$ is derivative of $-\log(f)$ with respect to the parameters and it is defined to be the score function.
It is noted that $-\log(f)$ can be regarded as an objective function $\rho$ in the context of robustness.
Here, the components of the local robustness of ML estimators are introduced.\\

\begin{definition}
Let $\psi$ be a score function derived by derivative of $-\log(f)$ with respect to parameters estimated. If $-\log(f)$ is differentiable,
then the function $\psi$ will be defined to be a score function of parameters \cite{Hampel86}.
\end{definition}

We will examine whether or not the score functions are bounded. When the limit of score function $\psi_{\theta}$ with respect to $x$ is bounded, ML estimators will be robust.\\

The functional forms of the score function corresponding to the parameters $\mu,\sigma,c,k,\varepsilon$ are
\begin{equation}\label{psimu}
  \psi_{\mu}(x)=\frac{c+1}{x}-\frac{(k+1)c(\frac{sign(x)x}{1+sign(x)\varepsilon})^{-c}}{x(1+(\frac{sign(x)x}{1+sign(x)\varepsilon})^{-c})},
\end{equation}

\begin{equation}\label{psisigma}
  \psi_{\sigma}(x)=c-\frac{(k+1)c(\frac{sign(x)x}{(1+sign(x)\varepsilon)})^{-c}}{1+(\frac{sign(x)x}{1+sign(x)\varepsilon})^{-c}},
\end{equation}

\begin{equation}\label{psic}
  \psi_{c}(x)= \frac{1}{c}-\log(\frac{sign(x)x}{1+sign(x)\varepsilon})+\frac{(k+1)(\frac{sign(x)x}{1+sign(x)\varepsilon})^{-c}\log(\frac{sign(x)x}{1+sign(x)\varepsilon})}{1+(\frac{sign(x)x}{1+sign(x)\varepsilon})^{-c}},
\end{equation}

\begin{equation}\label{psik}
  \psi_{k}(x)=\frac{1}{k}-\log\bigg(1+(\frac{sign(x)x}{1+sign(x)\varepsilon})^{-c} \bigg)
\end{equation}
\noindent and
\begin{equation}\label{psie}
  \psi_{\varepsilon}(x)=\frac{sign(x)(c+1)}{1+sign(x)\varepsilon}-sign(x)\frac{(k+1)c(\frac{sign(x)x}{1+sign(x)\varepsilon})^{-c}}{(1+sign(x)\varepsilon)(1+(\frac{sign(x)x}{1+sign(x)\varepsilon})^{-c})}.
\end{equation}

$\underset{x \rightarrow \pm \infty}{lim} \psi_\mu(x) = 0$, $\underset{x \rightarrow \pm \infty}{lim} \psi_\sigma(x) = c$,
$\underset{x \rightarrow \pm \infty}{lim} \psi_c(x) =  \infty$, $\underset{x \rightarrow \pm \infty}{lim} \psi_k(x) = 1/k$
and $\underset{x \rightarrow \pm \infty}{lim} \psi_\varepsilon(x) = \pm \frac{c+1}{1+sign(x)\varepsilon}$.
It is seen that the limit values of score functions for $\mu,k$ and $\varepsilon$ are finite,
but the limit value of $\psi_c(x)$ is infinite. Since the norm of influence function
$IF_{\mu,\sigma,k,\varepsilon}(x)$ of ML estimators for the parameters $\mu,\sigma,k$ and $\varepsilon$ is
bounded when $c$ is fixed and finite, the ML estimators for them are locally robust.
This shows us that the parameter $c$ must be fixed and finite. \cite{ArslanGenc09,Acitas13} and \cite{Lucas97} give the same discussion on the boundness of score function.



\begin{definition}
Let $\psi$ be a score function. Suppose that there is a point $x_0$. For  $0<x \leq x_0$, $\psi$ is (weakly) increasing.
For $x_0 < x < \infty$, $\psi$ is (weakly) decreasing. Then, $\psi$ is defined to be a redescending function \cite{Huber84}.
\end{definition}

\begin{enumerate}
\item $\rho(0)=0$  \cite{Huber84,ZhangLi98}.
\item $\underset{|x| \rightarrow \infty}{lim} \rho (x) = \infty$  \cite{Huber84,ZhangLi98}.
\item $\underset{|x| \rightarrow \infty}{lim} \frac{\rho (x)}{|x|} = 0$ \cite{Huber84}.
\item  For  $0<x \leq x_0$, $\psi$ is weakly increasing. For $x_0 < x < \infty$, $\psi$ is weakly decreasing. There is a such a point $x_0$  \cite{Huber84}.
\end{enumerate}
$\rho_{ESBIII}(x)=-\log(f)$. Let us show whether or not these conditions are satisfied.
\begin{enumerate}
\item $\rho_{ESBIII}(0)=0$ is not satisfied.
\item $\underset{|x| \rightarrow \infty}{lim} \rho_{ESBIII} (x) = \infty$ is satisfied.
\item $\underset{|x| \rightarrow \infty}{lim} \frac{\rho_{ESBIII} (x)}{|x|} = 0$ is satisfied.
\item This condition is satisfied for only the  MLE of parameter $\mu$.
\end{enumerate}
The condition $4$ is approved by the following lines.

It should be given some few comments. The derivatives of score functions with respect to $x$ for the parameters
$s,c,k,\varepsilon$ can not be solved according to $x$. Thus, it is not possible to find an exact
point $x_0$. From this result, it is said that the score functions of  parameters $s,c,k,\varepsilon$
 can not be in redescending function. Since a point $x_0$ can not be gotten,
it is not possible to imply that these functions are redescending.

\begin{proposition}
$\psi_\mu$ is a score function for parameter $\mu$. It can be possible to find a point $x_0$ at least for $\psi_\mu$.
$x_0={\frac {1}{a} \left( \,{\frac {-{c}^{2}k-{c}^{2}-ck+c+2+\sqrt {{c}^
{4}{k}^{2}+2\,{c}^{4}k+2\,{c}^{3}{k}^{2}+{c}^{4}+{c}^{2}{k}^{2}-2\,{c}
^{3}-2\,{c}^{2}k-3\,{c}^{2}}}{2(ck-1)}} \right) ^{-{c}^{-1}}}$,
\noindent where $a=\frac{sign(x)}{1+sign(x)\varepsilon}$.
\end{proposition}

Thus, it is possible to imply that the function $\psi_\mu$ is redescending.

\subsection{Tail behaviour of ESBIII distribution}

To call a heavy-tailed distribution or tail behaviour of a distribution, some definitions are proposed in \cite{Heavytail}:


\begin{definition}\label{def2}
 Let $\bar{F}(x)$ be $1-F(x)$. If $\underset{x \rightarrow + \infty}{lim} exp(\lambda x)\bar{F}(x) = \infty$ for all $\lambda >0$, then $F(x)$ is a heavy-tailed.
 \end{definition}

From equation (5), the CDF is 
$\frac{1-\varepsilon}{2}+(\frac{1+\varepsilon}{2})(1+(\frac{x}{1+\varepsilon})^{-c})^{-k}$ for $x \geq 0$. Here, the part $\frac{1-\varepsilon}{2}$ is negative side of PDF. This part is due to variable transformation. It leads to asymmetrization on a symmetric function. Thus, it can be hold on for the representation of negative side of PDF. From this clarification,
$(\frac{1+\varepsilon}{2})(1+(\frac{x}{1+\varepsilon})^{-c})^{-k}$ 
is a part for the positive side of CDF. In such a case, 
$\bar{F}(x)=\frac{1+\varepsilon}{2}-(\frac{1+\varepsilon}{2})(1+(\frac{x}{1+\varepsilon})^{-c})^{-k}$. 
Let us show the $\underset{x \rightarrow + \infty}{lim} exp(\lambda x)\bar{F}(x)$ for the ESBIII. As it is seen, 
$\underset{x \rightarrow + \infty}{lim} exp(\lambda x)\bar{F}(x)=\infty$. 
As a result, this property for heavy-tailedness is satisfied.


Since $ESBIII$ distribution can be considered as a heavy-tailed distribution, the informative behaviour can be sensible degree when the data set includes one or more outliers.
In the following section, real data examples will be considered. Some useful comments can be given in the following lines.

As it is seen from Figure $1$, the platikurtic-leptokurtic property is hold for the ESBIII.
Thus, ESBIII can be a good candidate for the data modelling when there is a comparison test
between ESBIII and ESEP. When we look at the advantage of ESBIII, it is observed that ESBIII
can be preferable in the context of CDF performance, because  the CDF of ESEP distributions
depends on the incomplete gamma function. For the CDF of ESBIII, we have an explicit analytical form,
which does make that KS test statistic of ESBIII is free from the computational error occurred by the
special function that is incomplete gamma function. The second derivative test is not examined for
the parameters of ESBIII distribution, because the analytical form of PDF is not easy to get the Hessian
matrix and its eigenvalues to check that the eigenvalues are negative for the parameters $c >0$ and $k>0$.
However, there is a good way to overcome this problem, because ESBIII has an explicit analytical form for
CDF, in other words, it is free from a special function such as incomplete gamma function. Thus,
this explicit analytical form will help us to control the efficient fitting performance of the ESBIII via KS test statistic.

%
%

\section{Real Data Application}\label{applications}
In this section, two real data sets will be used to illustrate the modeling capability and robustness property of the proposed distribution.


The bimodal distributions having the CDF are considered. Some of them are the $\varepsilon$-skew exponential distribution, its bimodal version with different height  ($BEGG(\alpha,\beta,$ $\delta_0, \delta_1,\eta,\varepsilon,\mu,\sigma)$) considered by \cite{Elsalloukhetal05,Cankayaetal15}, respectively. The special case of the $\varepsilon$-skew exponential power distribution is the $\varepsilon$-skew Laplace distributions ($ESL$). The detailed discussion about the special cases of $BEGG$ distribution can be found in \cite{Cankayaetal15}. The generalized $t$ distribution in the sense of $\varepsilon-$skew form considered by \cite{ArslanGenc09} ($ESGT(p,q,\varepsilon,\mu,\sigma)$), the $\alpha$-skew Laplace distribution ($ASL(\alpha,\mu,\sigma)$) considered by \cite{Sham13} and the exponentiated sinh cauchy distribution ($ESC(\lambda,\beta,\mu,\sigma)$) considered by \cite{Cooray13} are used to model the data set.


\textbf{Example 1}: \cite{Arslan09a,Arslan09b} modelled the univariate cases of the data sets in cDNA microarray. They proposed parametric models. These models are used to fit the data sets. This data sets were also considered by \cite{Acitas13,PurdomHol05} for modelling data sets in the univariate case. In this study, we consider to model the cDNA microarray datasets. In these data sets, AT-matrix: 1375 genes $\times$ 118  drugs which are available at the web site
\begin{verbatim}http://discover.nci.nih.gov/
nature2000/data/selected_data/at_matrix.txt\end{verbatim}
are considered to test the performance of $ESBIII$.\\

In these data sets, the column one of gene-drug correlation data, which is named as "sid 487493, erythrocyte band 7
integral membrane protein [5,3:aa045112]", will be modelled by means of ESBIII distribution.
The candidates of $\varepsilon$-skew distribution family (ESEP) were proposed by \cite{Elsalloukhetal05} and \cite{Mud00}.
The scale mixture form of ESEP was considered by \cite{ArslanGenc09}, which can be considered as a $\varepsilon$-skew generalized $t$ distribution.
In this framework, we want to make a comparison among the distributions used in this work.

\begin{tiny}
\begin{table}[htb]\label{esepesb3ex1mnc}
\begin{center}
\caption{Parameter estimates for the microarray data set for the different parametric models}
\begin{tabular}{c c c c c c c}
\hline Parametric Models &$\hat{\mu}$ & $\hat{\sigma}$  & $\hat{\varepsilon}$ & $P(KS)$ & $AIC$  \\ [0.5ex]  
\hline \hline 
ESBIII ($\hat{c}= 2.3826, \hat{k}= 0.7786$)  &  -0.0061 &  0.0770 & 0.0533  & 0.8478 & -432.4893 \\
ESL ($\alpha=1$) & 0.0130 & 0.0770 & 0.0484 & 0.0838 & -352.1553 \\
ESGT($p=1,q=3/4$)  & 0.0440 & 0.0549 & -0.2773 & 0.0540 &   -107.9998 \\
\begin{tiny} BEGG($\hat{\alpha}=0.9824,\hat{\beta}=1,\hat{\delta_0}=0.0001,\hat{\delta_1}=0.4622,\hat{\eta}=0.7277$) \end{tiny} & 0.0105 & 0.0701 & 0.2301 & 0.2900 & -82.5059 \\
ASL($\hat{\lambda}=0.0781)$ & 0.0193 & 0.1022      &- & 0.5321  & -173.2460  \\
ESC($\hat{\lambda}=2.3591, \hat{\beta}=0.944) $ & 0.0051 & 0.1522   &- & 0.4037 &  -158.4236 \\
 \hline
\end{tabular}
\label{sheeptable}
\end{center}
\end{table}\end{tiny}


\begin{figure}[!htb]
\centering
 \begin{subfigure}{.55\linewidth}
    \includegraphics[width=.85\textwidth]{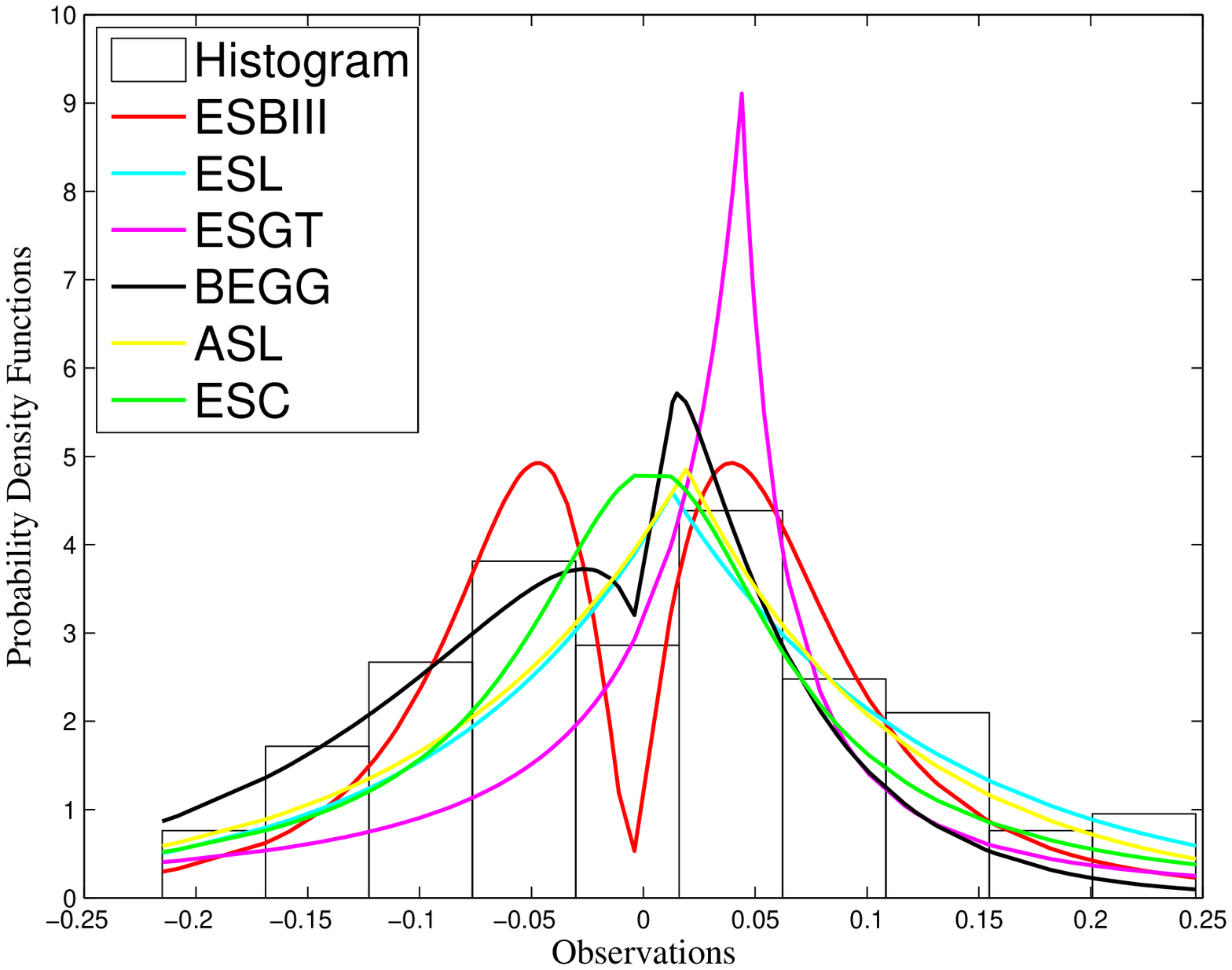}
 \caption{Histogram of the data with the fitted PDF functions for microarray data}
  \end{subfigure}%
 \begin{subfigure}{.55\linewidth}
    \includegraphics[width=.85\textwidth]{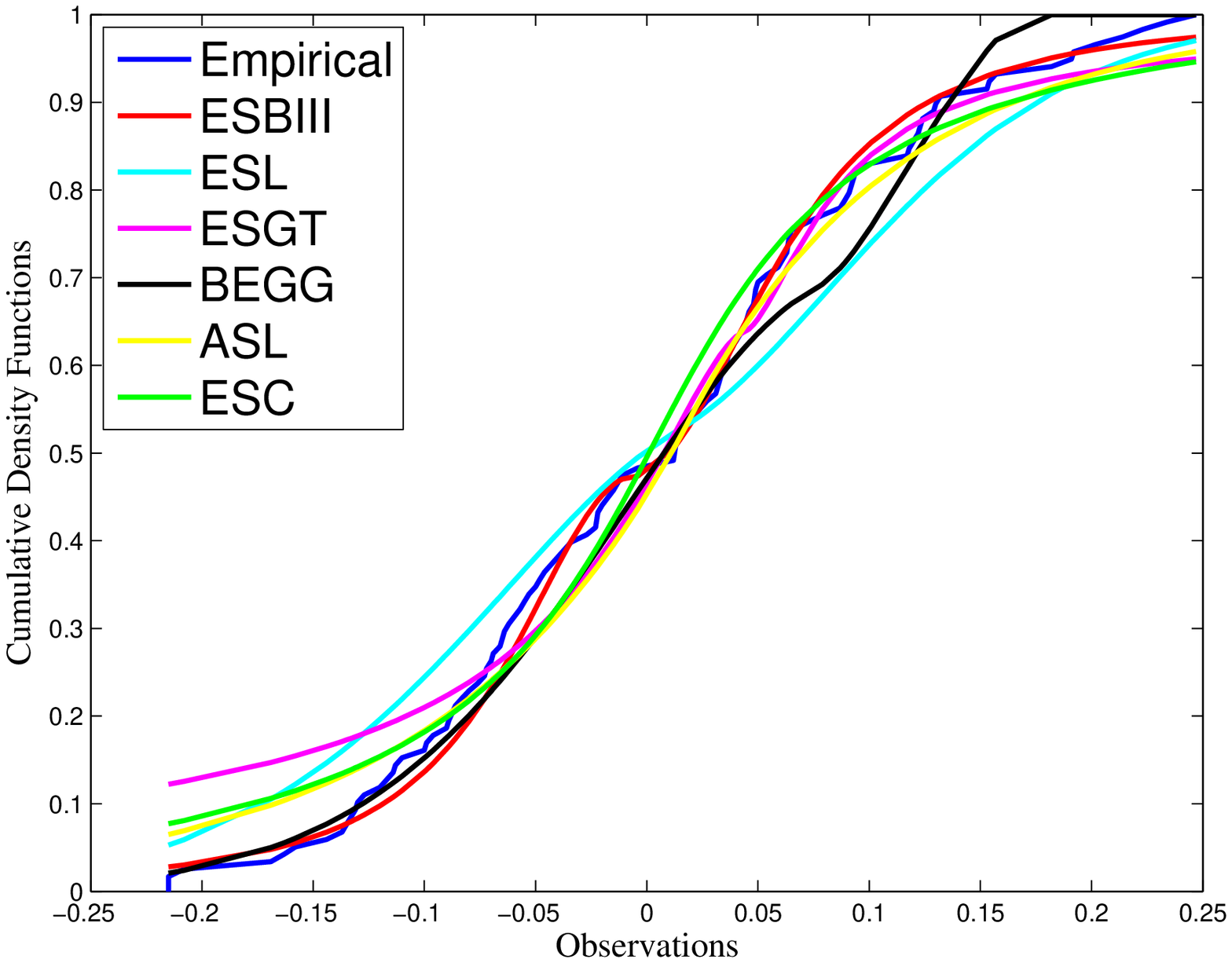}
 \caption{Empirical CDF of the data with the fitted CDF functions for microarray data}
  \end{subfigure}
\end{figure}

\textbf{Example 2}: Martin Marietta data set is very popular in the literature. In this work, we will focus on fitting the explanatory
variable in the regression model. This variable is defined to be the excess rate of the Martin Marietta company.

\begin{tiny}
\begin{table}[htb]\label{esepesb3exMmnc}
\begin{center}
\caption{Parameter estimates for the Martin Marietta data set for the different parametric models}
\begin{tabular}{c c c c c c c}
\hline Parametric Models &$\hat{\mu}$ & $\hat{\sigma}$  & $\hat{\varepsilon}$ & $P(KS)$ & $AIC$  \\ [0.5ex]  
\hline \hline 
ESBIII ($\hat{c}=  2.1241, \hat{k}=  0.7190$)  &  -0.0005 &  0.0652 & 0.0564  & 0.6136 & -242.7349 \\
ESL ($\alpha=1$) & -0.0050 & 0.0652 & -0.1434 & 0.1529 & -195.2337 \\
ESGT($p=1.2,q=1/4$)  & -0.0096 & 0.0428 &  0.0699 &  0.4682 &   -32.3817 \\
\begin{tiny} BEGG($\hat{\alpha}=1.0824,\hat{\beta}=1.02,\hat{\delta_0}=0.2528,\hat{\delta_1}=0.0090,\hat{\eta}=0.7105$) \end{tiny} & 0.0201 & 0.0615 & 0.0301 & 0.1111 & -11.5659 \\
ASL($\hat{\lambda}=0.0001)$ & 0.0188 & 0.1018      &- & 0.5321  &  -90.4340  \\
ESC($\hat{\lambda}=2.5921, \hat{\beta}=0.9860) $ & 0.0139 & 0.1534   &- & 0.4283 &  -88.4240 \\
 \hline
\end{tabular}
\label{sheeptable}
\end{center}
\end{table}\end{tiny}

\begin{figure}[!htb]
\centering
 \begin{subfigure}{.55\linewidth}
    \includegraphics[width=.85\textwidth]{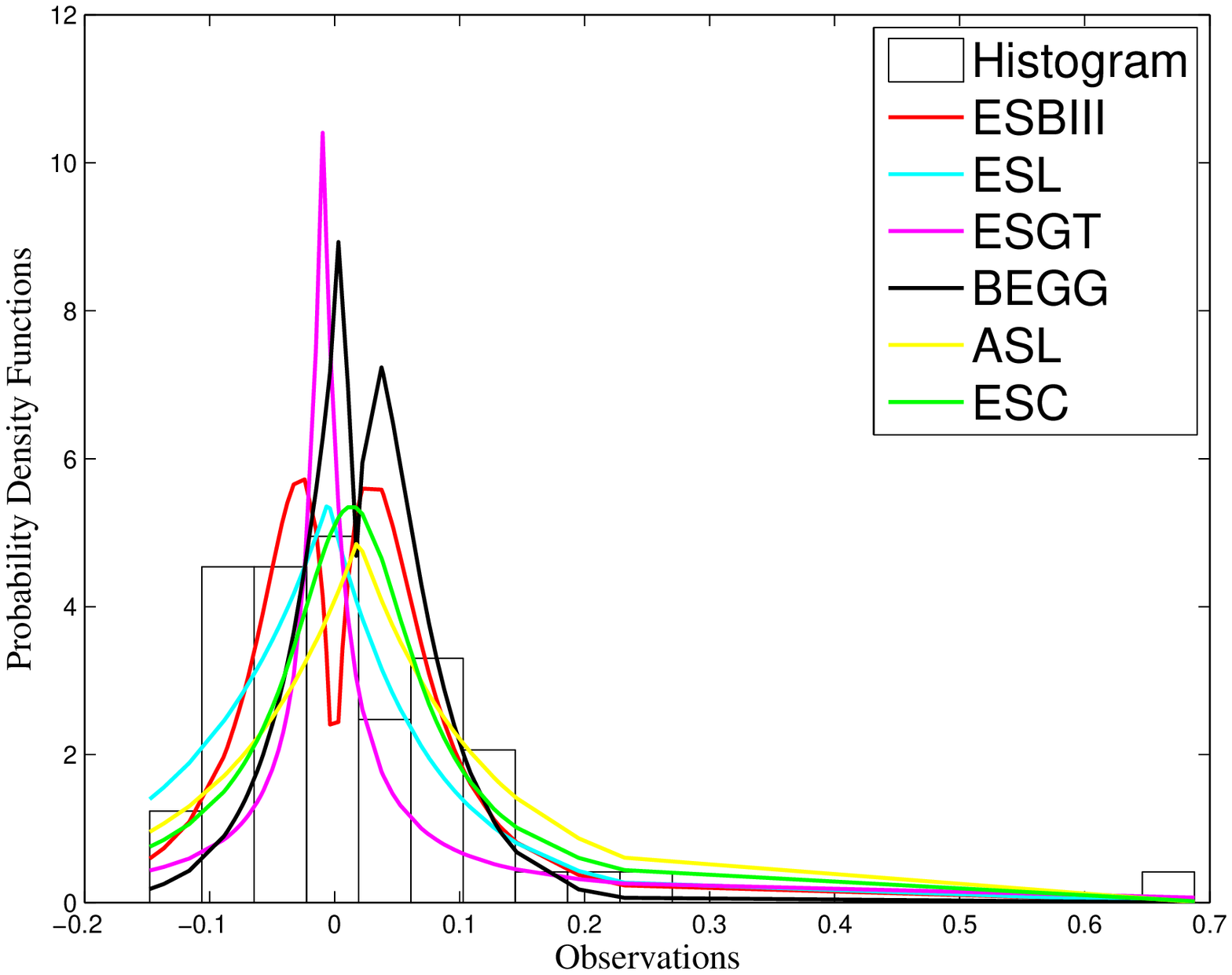}
 \caption{Histogram of the data with the fitted PDF functions for Martin Marietta data}
  \end{subfigure}%
 \begin{subfigure}{.55\linewidth}
    \includegraphics[width=.85\textwidth]{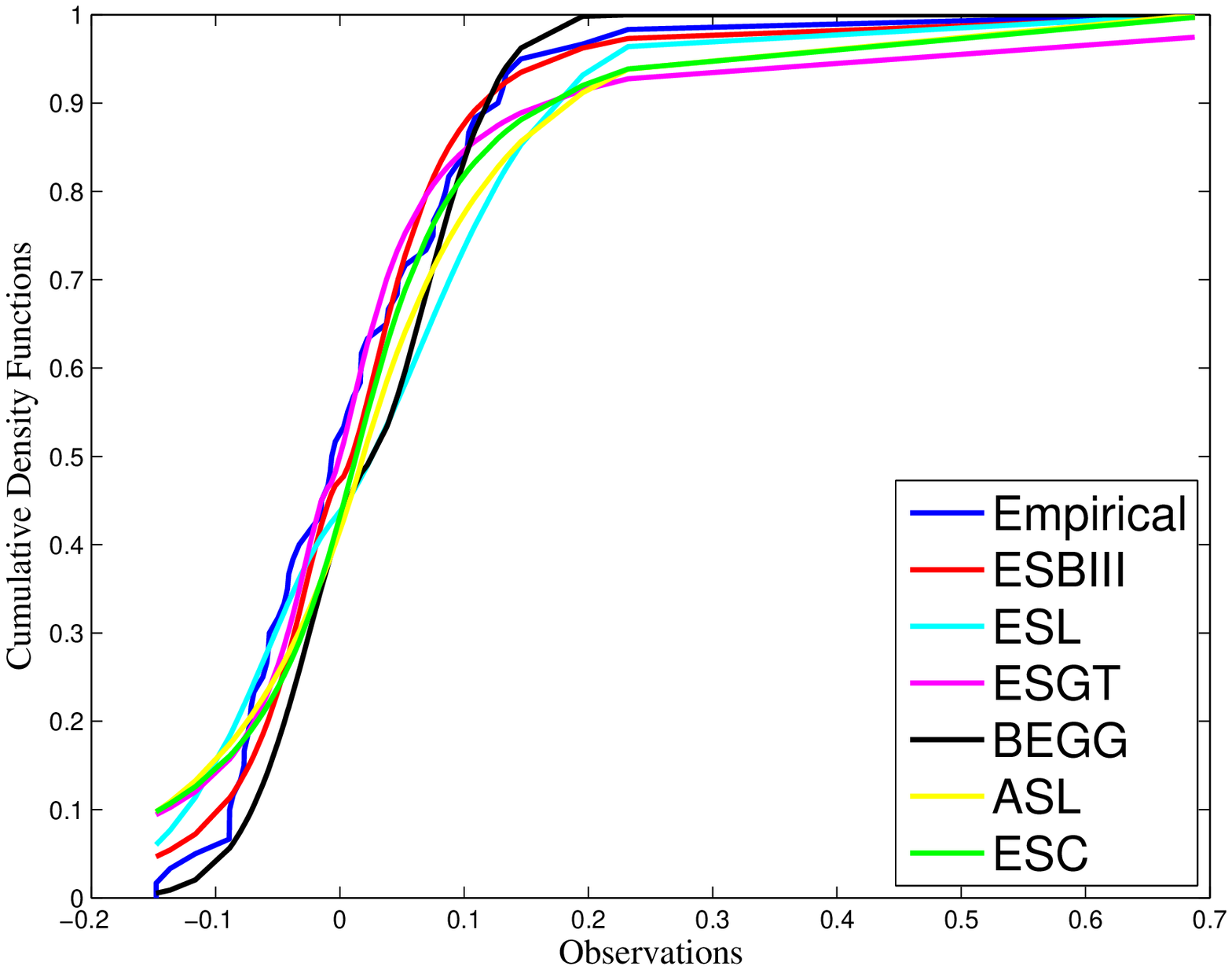}
 \caption{Empirical CDF of the data with the fitted CDF functions for Martin Marietta data}
  \end{subfigure}
\end{figure}

Tables 2  and 3 show that ESBIII distribution fits better than the $\varepsilon$-skew candidate ones, such as ESL, ESGT, BEGG, ASL and ESC distributions.
It is considered that the parameters $c$ and $k$ produce more flexibility when compared with ESL, ESGT, BEGG, ASL and ESC especially.
The parameters $c$ and $k$ of ESBIII produce the flexible behaviour for the peakedness and tail thickness both together.
In addition to this property, the interspace between peakedness and tail of ESBIII can also be flexible,
because we have two shape parameters $c$ and $k$. This property can be considered as an advantage for modelling capacity of the distribution. \\

Tables 2 and 3 show the estimated values of parameters in the used models, $P(KS)$ that is the probability value (p-value) of KS statistic and
Akaike information criterion (AIC) values. It is known that the more p-value of KS statistic is highest, the more modelling capacity is best.
When the parametric models are ordered according to the modelling capacity from highest to lowest one,
ESBIII, ASL, ESC, BEGG, ESL and ESGT distributions are given in order for microarray data.
ESBIII distribution has the high p-value of KS test. Why it can model the data very efficiently is that
the function can capture the frequency as a flexible way. It can be considered that the parameters $c$ and $k$ keep their property that is shapeness.

When the parametric models are ordered according to the modelling capacity from highest to lowest one,
ESBIII, ASL, ESGT, ESC, ESL and BEGG distributions are given in order for Martin Marietta data.
Since ESBIII is in the class of the heavy-tailed distributions, it can model the data in tail.
The bimodality property of ESBIII distribution plays an important role in fitting the data set as well.
Same interpretation on modelling capacity can not be said for the other distributions.
As it is seen, the bimodality and heavy-tailedness are advantage for fitting.

\cite{Lucas97} and \cite{ArslanGenc09} prefer to use the fixed values of parameters $p$ and $q$ as a tuning constant in robustness,
because the ML estimators of the parameters are not robust owing to fact that the IF of them is unbounded. For our case, since the
score function for the parameter $c$ is infinite, IF of ML estimators of parameters will be unbounded. We prefer to estimate it,
because it is clearly seen that the p-value of KS statistic shows the competence of our flexible distribution.
Using both the robustness of ML estimators for the parameters of distribution and the goodness of fit test representing the
modelling capacity of proposed distribution together is suggested. For our real data examples,
examining whether or not the p-value of KS statistic is high or the value of AIC is small will be a good indicator to
approve the modelling capacity of proposed distribution. Since the goodness of fit tests (KS and AIC) gives the satisfied
results, the parameter $c$ was also estimated. In fact, IF of ML estimators for the parameters
$\mu$, $\sigma$, $c$, $k$ and $\varepsilon$ is unbounded, however there is another point it should be taken in account,
which is that the modelling capacity of the proposed model is very high when the p-values of KS statistic are considered.
If the parameter $c$ was considered to be fix, the estimated value for it would be wonderful choice, because ESBIII has a
high p-value of KS statistic, which means that the high KS value or small AIC value show the best fitting performance on data set.
The robustness property of ESBIII also supports the performance of ESBIII.

\section{Conclusions}\label{conclusions}

In this paper, an extension to real axis for the Burr III distribution has been proposed. The main goal at supplying the new distribution is that Burr III distribution has heavy tailed property. An extension to real axis already preserves the heavy tailed (robustness) property of this distribution. A skewness parameter was also added to the extended Burr III. The skewness parameter provides another flexibility when the positive side and negative side of real axis have different frequency. In this context, ESBIII distribution is one in the class of the heavy-tailed candidates and has a skewness parameter as well.

The parameters $c$ and $k$ in Burr III distribution are the shape parameters, which gives the flexibility on function. It is readily seen that some values of them produce unimodal, bimodal, leptokurtic and platikurtic distributions. The skewness and the kurtosis values of ESBIII distribution for the different values of parameters also approve the competency on flexibility of the proposed distribution. As a result, the ESBIII copes with skewness, heavy-tailedness, bimodality by means of just only three parameters.


The some properties of ESBIII were introduced; such as the R\'{e}nyi entropy and the $rth$ moment of X and the characteristic function were calculated.
The ML estimators of the parameters in ESBIII have been obtained.
The local robustness property of ML estimators for the parameters in ESBIII was examined and also ESBIII was examined to see the heavy-tailedness property.
Since we have the ML estimators for the parameters in ESBIII, the real data examples can be applied to see the performance of the distribution.
For this purpose, two real data examples were given. It was observed that the p-values of KS statistic for ESBIII were very high for the considered data sets.
We can assert that having shape parameters $c$ and $k$ which can work both together to construct the peakedness and the tail behaviour of the distribution produces
a flexible distribution. The skewness parameter $\varepsilon$ helps to model the uneqaul frequency in positive and negative part of real axis.
In other words, adding the skewness parameter $\varepsilon$ helps to generate a model which can capture the skewness in data.
For another study, one can make a comparison between the flexibility performance of Azzalini type skewness of the extended Burr III
distribution and $\varepsilon$ skew version proposed in this study. In ongoing paper, Shannon, Tsallis, R\'{e}nyi and hybrid entropies of distributions
\cite{JizbaKorbel16,CankayaKorbel17} will be used as a tool to test how much the degree of heavy-tailedness can be produced by some distributions.

\end{document}